\documentclass[10pt]{amsart}

\usepackage{latexsym}
\usepackage{amsmath}
\usepackage{amssymb}
\usepackage{amscd}
\usepackage{amsthm}
\usepackage{xypic} 
\usepackage{ifthen} 
\usepackage{hyperref} 

\theoremstyle{plain}
 \newtheorem{theorem}{Theorem}[section]
 \newtheorem{proposition}[theorem]{Proposition}
 \newtheorem{lemma}[theorem]{Lemma}
 \newtheorem{corollary}[theorem]{Corollary}
 
\theoremstyle{definition}
 \newtheorem{definition}[theorem]{Definition}

\theoremstyle{remark}
 \newtheorem{remark}[theorem]{Remark}

\makeindex

\begin{document}
\title[On the rationality of the moduli space of L\"uroth quartics]{On the rationality of the moduli space of L\"uroth quartics} 

\author{Christian B\"ohning\\ 
 Hans-Christian Graf von Bothmer}

\maketitle

\newcommand{\PP}{\mathbb{P}} 
\newcommand{\QQ}{\mathbb{Q}} 
\newcommand{\ZZ}{\mathbb{Z}} 
\newcommand{\CC}{\mathbb{C}} 
\newcommand{\rmprec}{\mathrm{prec}}
\newcommand{\rmconst}{\mathrm{const}}
\newcommand{\xycenter}[1]{\begin{center}\mbox{\xymatrix{#1}}\end{center}} 

\newboolean{xlabels} 
\newcommand{\xlabel}[1]{ 
                        \label{#1} 
                        \ifthenelse{\boolean{xlabels}} 
                                   {\marginpar[\hfill{\tiny #1}]{{\tiny #1}}} 
                                   {} 
                       } 
\setboolean{xlabels}{false} 

\begin{abstract}
We prove that the moduli space $\mathfrak{M}_L$ of L\"uroth quartics in $\PP^2$, i.e. the space of quartics which can be circumscribed
around a complete pentagon of lines modulo the action of $\mathrm{PGL}_3 (\CC )$ is rational, as is the related moduli space of Bateman
seven-tuples of points in $\PP^2$.
\end{abstract}

\section[Introduction]{Introduction}
The $13$-dimensional hypersurface in the parameter space $\PP^{14}$ of plane quartics consisting of the closure of the locus of quartics which
have an inscribed complete pentagon of lines, the L\"uroth quartics, is classically called the L\"uroth hypersurface and known to be irreducible,
$\mathrm{SL}_3 (\CC )$-invariant, and of degree $54$. In this paper we prove the rationality of the moduli space of L\"uroth quartics, i.e. of
the corresponding $\mathrm{SL}_3 (\CC )$-quotient of the L\"uroth hypersurface. Morley \cite{Mor} described L\"uroth quartics as branch curves of
degree $2$ rational self-maps of $\PP^2$ called Geiser involutions. These involutions are determined by a net of cubics in $\PP^2$ through seven
distinct points in $\PP^2$ no six of which lie on a conic. Thus we get a map from such seven tuples of points to quartics, and certain configurations of points, called \emph{Bateman configurations}, lead to L\"uroth quartics. Morley proves that
there are generically $8$ Bateman configurations leading to a L\"uroth quartic. It is shown below that also the moduli space of Bateman seven
tuples of points is rational. 

\

We would like to thank Edoardo Sernesi for bringing these interesting geometric questions to our attention and for helpful comments
on the paper.

\section[L\"uroth quartics and Bateman points]{L\"uroth quartics and bateman points}

\begin{definition}\xlabel{dLuroth}

A quartic curve $L \subset \PP^2$ is called a {\sl L\"uroth quartic} if there exist $ l_1, \dots , l_5 \in (\CC^3)^{\vee }$ such that $\cup_{i \neq
j } ((l_i = 0) \cap (l_j = 0))$ consists of $10$ points contained in $L$. We denote by $\mathfrak{P}_L$ the subset inside the parameter
space $\PP^{14}$ of plane quartic curves corresponding to L\"uroth quartics. 

A quartic curve $C$ is called a {\sl Clebsch quartic} if there exist $ l_1, \dots , l_5 \in (\CC^3)^{\vee }$ such that $C = \{ l_1^4 + \dots + l_5^4
= 0 \}$. As in the previous item we denote by $\mathfrak{P}_C$ the subset of Clebsch quartics inside the space $\PP^{14}$ of all
plane quartics.

\end{definition}

There is an $\mathrm{SL}_3 (\CC )$-equivariant map $S_4 \, :\, \mathfrak{P}_C \to \mathfrak{P}_L$ called the \emph{Scorza map} which
 may be written symbolically as
\[
S_4 = (\alpha \beta \gamma ) (\alpha \beta \delta ) (\alpha \gamma \delta ) (\beta \gamma \delta ) \alpha_x \beta_x \gamma_x \delta_x \, .
\]
We refer to \cite{G-Y} for an account of the symbolical notation. It is known (cf. \cite{Dol2}, section 6.4.1) that $S_4$ induces a
generically $1:1$ map
\[
S_4 \, : \, \mathfrak{P}_C \to \mathfrak{P}_L
\]
so that $\mathfrak{M}_L = \mathfrak{P}_L / \mathrm{SL}_3 (\CC )$ and $\mathfrak{M}_C = \mathfrak{P}_C /\mathrm{SL}_3 (\CC )$ are
birational. 

\begin{theorem}\xlabel{tMain}
The space $\mathfrak{M}_C$, hence also $\mathfrak{M}_L$, is rational.
\end{theorem}

\begin{proof}
By \cite{Dol2}, section 6.4.1, $\mathfrak{P}_C$ is an open subset of the locus $\mathcal{K} \subset \PP^{14} = \PP (\mathrm{Sym}^4
(\CC^3)^{\vee })$ given by the vanishing of the catalecticant invariant of quartics which is a degree $6$ polynomial in the coefficients
$A_{\underline{i}}$ of the quartic and defined as follows: for a quartic $C$ it is the determinant of the linear map
\begin{gather*}
l_C \, : \, \mathrm{Sym}^2 (\CC^3 )  \to \mathrm{Sym}^2 (\CC^3)^{\vee }
\end{gather*}
given by acting on $C = ( ax_0^4 + bx_0^3x_1 + \dots )$ with a dual conic $Q \in  \mathrm{Sym}^2 (\CC^3 )$, $Q = A
\frac{\partial^2}{\partial x_0^2 } + B \frac{\partial^2}{\partial x_0 x_1 } + \dots $, where we use dual bases in $\mathrm{Sym}^2 (\CC
)^3$ and $\mathrm{Sym}^2 ((\CC^3)^{\vee }) \simeq (\mathrm{Sym}^2 (\CC^3))^{\vee }$. We may view $l_C$ as induced by the multiplication
map
\[
m \, : \, \mathrm{Sym}^2 (\CC^3) \times \mathrm{Sym}^2 (\CC^3) \to \mathrm{Sym}^4 (\CC^3) \, .
\]
Thus in summary $\mathfrak{M}_L$ is birational to $\mathcal{K} /\mathrm{SL}_3 (\CC )$. By its definition, $\mathcal{K}$ fibres in an
$\mathrm{SL}_3 (\CC )$-equivariant way:
\begin{gather*}
\begin{CD}
\mathcal{K}\\
@V{\pi }VV \\
\PP (\mathrm{Sym}^2 (\CC^3))
\end{CD}
\end{gather*}
associating to a generic $C$ the kernel $\mathrm{ker} (l_C)$. The space $\PP (\mathrm{Sym}^2 (\CC^3))$ has a $(\mathrm{PSL}_3 (\CC ) , \:
\mathrm{SO}_3 (\CC ))$-section, namely a point. Hence $\mathcal{K} / \mathrm{SL}_3 (\CC )$ is birational to $F/\mathrm{SO}_3 (\CC )$
where $F$ is the generic fibre of $\pi $. In view of the isomorphism $\mathrm{SO}_3 (\CC ) \simeq \mathrm{PSL}_2 (\CC )$, it remains to
see that
\begin{gather}
F \simeq \PP ( \mathrm{Sym}^4 (\CC^3) / \mathrm{Sym}^2 (\CC^3) ) \simeq \PP ( V(8))
\end{gather}
as $\mathrm{PSL}_2 (\CC )$-space where $V(n)$ is the representation of this group on the space of binary forms of degree $n$ for
$n\in\mathbb{N}$. Note $\mathrm{Sym}^4 (\CC^3) \simeq V(8) \oplus V(4)  \oplus V(0)$, $\mathrm{Sym}^2 (\CC^3) = V(4)  \oplus V(0)$.\\
For once (1) is established we get the rationality of $\mathfrak{M}_L$ from that of $\PP (V(8)) /\mathrm{PSL}_2 (\CC )$ (for a proof of
the latter see \cite{Bogo2}, \cite{Kat83}, \cite{Kat84}).\\
To see that (1) holds note that the fibre over $[Q]$ of $\pi$ consists of those quartics $C$ such that $Q$ is in the kernel of $l_C$, or
equivalently of those linear forms on $\mathrm{Sym}^4 (\CC^3)$ which vanish on the image of $\mathrm{Sym}^2 (\CC^3)$ inside
$\mathrm{Sym}^4 (\CC^3)$ under the map $m( Q, \cdot )$. Hence the claim.
\end{proof}

We turn to the proof of rationality for the moduli space of Bateman $7$-tuples of points, see \cite{O-S} for background information.

\

Consider $2\times 3$-matrices
\begin{gather}\label{points}
\left( \begin{array}{ccc} l_1 & l_2 & l_3 \\ q_1 & q_2 & q_3 \end{array}  \right)
\end{gather}
where the $l_i$ are linear, and the $q_i$ quadratic forms on $\PP^2$, and remark that seven general points in $\PP^2$ define a matrix
of type (2) uniquely as syzygy matrix, and conversely a general matrix of type (\ref{points}) determines seven points in
$\PP^2$ as its rank $1$ locus. This can be made precise by defining $\mathbb{M}$ as the projective space of all matrices of type (\ref{points})
which is acted on by the group
\[
\mathcal{G} = \mathrm{SL} (V) \times \mathrm{SL} (W) \times G
\]
where $\PP^2 = \PP (V)$ and $\mathrm{SL} (V)$ corresponds to coordinate changes, $W\simeq \CC^3$ is the space of columns so that $\mathrm{SL}
(W)$ corresponds to column operations, and $G$ is the group of row operations (adding a linear form times the
first row to the second row). Then $\mathbb{M} / \mathcal{G}$ is a birational model for the moduli space of seven-tuples of points in $\PP^2$.

\

Now seven general points in $\PP^2$ can also be obtained as zero locus of a general global section in $\mathcal{T}_{\PP^2} (1)$,
$\mathcal{T}_{\PP^2}$ the tangent bundle on $\PP^2$, and by Bott's theorem $H^0 (\mathcal{T}_{\PP^2} (1))= V(1, 2)$, denoting by $V(a, \: b)$ the
irreducible representation of $\mathrm{SL} (V)$ whose highest weight has numerical labels $a,
\: b$ (so e.g. $V(1, \: 0) =V$, $V(0, \: 1) = V^{\ast }$). Thus we may also view $V(1, 2)$ as a parameter space of seven points in $\PP^2$ and
we can relate this to the preceding matrix picture as follows: the $\mathcal{G}$-representation corresponding to the space of matrices of type
(\ref{points}) is
\[
V^{\ast } \otimes W + S^2 V^{\ast } \otimes W.
\]
Acting by an element of $\mathcal{G}$ we may achieve that a general matrix of type (\ref{points}) can be put in the
form 
\begin{gather}\label{normalform}
\left(
\begin{array}{ccc}
x & y & z \\
q_1 & q_2 & q_3
\end{array}
\right)\, .
\end{gather}
The stabilizer of matrices of this form inside $\mathcal{G}$ is isomorphic to $\mathrm{SL}(V) \times G$. This means that we now have an
isomorphism $\mathrm{SL} (V) \simeq \mathrm{SL} (W)$ and as $\mathrm{SL} (V)$-representation the full space of matrices decomposes
\begin{gather}\label{decomposition}
V^{\ast}\otimes V + S^2 V^{\ast } \otimes V = (V(0, \: 0) + V(1, \: 1)) + (V(1, \: 2) + V(0, \: 1))
\end{gather} 
where $V(0, \: 0)$ corresponds to our choice of $(x, \: y, \: z)$ in the first row. Acting with an element of the group of row
operations $G$, we see that the $\mathcal{G}$-action on the space of all matrices of type (\ref{points}) has an $\mathrm{SL} (V)$-section
which is $V(1, 2)$. Thus we see the birational isomorphism between $\mathbb{M} / \mathcal{G}$ and $\PP (V(1, 2)) /\mathrm{SL} (V)$ explicitly.

In other words, after the choice of the section (3), the space of matrices decomposes as in (4), and the points a
matrix defines may be seen as the $V(1,2 )$-component with respect to this decomposition.

\begin{definition}\xlabel{dBatemanpoints}
An unordered collection of seven points $P_1, \dots , P_7$ in $\PP^2$ is said to be a \emph{Bateman seven-tuple of points} in $\PP^2$ if it is
defined by a matrix of type (\ref{points}) where the $l_i$ are the partial derivatives of a quadratic form $Q$ on $\PP^2$ and the $q_i$ the partial
derivatives of a cubic form $C$ on $\PP^2$. We denote their moduli space for projective equivalence by $\mathfrak{M}_B$.
\end{definition}

\begin{remark}\xlabel{rMorleyCover}
Morley \cite{Mor} showed that $\mathfrak{M}_B$ is an eight to one cover of $\mathfrak{M}_L$ via taking the branch curve of the Geiser involution
determined by the Bateman points. See also \cite{O-S}.
\end{remark}

What is the representation-theoretic meaning of the concept of Bateman points?

\begin{lemma}\xlabel{lBateman}
There is a unique bilinear map 
\[
b \, : \, V(2, \: 0) \otimes V(0,\: 3) \to V(1, \: 2)
\]
and the image of the pure tensors in $V(2, \: 0) \otimes V(0,\: 3)$ under $b$ may be identified with the Bateman seven-tuples of points. More
precisely, let the Bateman seven-tuple be defined by $Q$ and $C$ as in Definition \ref{dBatemanpoints}. Let $Q^{\ast } \in V(2, 0)$ be the
quadric dual to $Q$ (recall that there is a degree $2$ equivariant map $V(0, 2) \to V(2, 0)$ which corresponds to taking the adjugate of a
symmetric matrix). Then $b(Q^{\ast }, C) \in V(1, 2)$ is the associated Bateman seven tuple, using the correspondence between matrices and point
sets encoded in the decomposition \ref{decomposition}. 
\end{lemma}

\begin{proof}
Let us start with $Q$ and $C$ and consider the unique equivariant maps (\emph{derivatives})
\begin{gather*}
d_2 \, :\, V(2, \: 0) \otimes V(0, \: 2) \to V(1,\: 1) + V(0, \: 0) , \\
d_3 \, :\, V(2, \: 0) \otimes V(0, \: 3) \to V(1, \: 2) + V(0, \: 1)\, .
\end{gather*}
Then the matrix of the partials of $Q$ and $C$ is, in the identification made in decomposition \ref{decomposition}, nothing but the pair 
\[
(d_2 (Q^{\ast } , \: Q ) , \; d_3 (Q^{\ast }, \: C)) \, .
\]
In fact one calculates that $d_2 (Q^{\astÊ}, Q) \in V(0,\: 0)$ (it suffices to check this for a particular quadric by equivariance). Note
that
\begin{gather*}
d_2 (Q^{\ast } , Q ) = \sum_{i=1 }^{3} \frac{\partial}{\partial e_i} (Q^{\ast }) \otimes \frac{\partial }{\partial x_i} (Q) , \\
d_3 (Q^{\ast } , C ) = \sum_{i=1 }^{3} \frac{\partial}{\partial e_i} (Q^{\ast }) \otimes \frac{\partial }{\partial x_i} (C)
\end{gather*}
where $e_i$ and $x_j$ are dual coordinates in $\CC^3$ and $(\CC^3)^{\ast }$. So the whole point of taking derivatives with respect to the quadric
$Q^{\ast }$ dual to $Q$ is to obtain a matrix in normal form \label{normalform} to achieve consistency with the choice of section made there.
This proves the Lemma.
\end{proof}

Now consider the isomorphism
\[
m \, : \, \Lambda^2 V(2, \: 0) \to V(2, \: 1)\, .
\]
This may be interpreted as a natural skew symmetric matrix $M$ on the space $V(2, \: 0)$ with entries linear forms on the space $V(1, \: 2)$ of
points. Now the map $V(2,\: 0)\otimes V(2, \: 0) \to V(2, \: 1)$ given by $m$ is just the basic map $\vartheta$ on page 252 of \cite{BvB10} whence in the
basis
\begin{gather*}
\left\{\frac{1}{2} {e}_{1}^{2},\: {e}_{1} {e}_{2},\: {e}_{1} {e}_{3},\: \frac{1}{2}
      {e}_{2}^{2},\: {e}_{2} {e}_{3},\: \frac{1}{2} {e}_{3}^{2}\right\}
\end{gather*}
of $V(2, \: 0)$ a computation with e.g. \ttfamily Macaulay 2\rmfamily $\:$yields
\begin{gather}\label{FormulaM}
M = \left(\begin{smallmatrix}0&
        {\delta}_{311}&
        {-{\delta}_{211}}&
        {\delta}_{321}&
        -{\delta}_{221}+{\delta}_{331}&
        {-{\delta}_{231}}\\
        {-{\delta}_{311}}&
        0&
        {\delta}_{111}-{\delta}_{221}-{\delta}_{331}&
        {\delta}_{322}&
        {\delta}_{121}-{\delta}_{222}+{\delta}_{332}&
        {\delta}_{131}-{\delta}_{232}\\
        {\delta}_{211}&
        -{\delta}_{111}+{\delta}_{221}+{\delta}_{331}&
        0&
        -{\delta}_{121}+{\delta}_{332}&
        -{\delta}_{131}-{\delta}_{232}+{\delta}_{333}&
        {-{\delta}_{233}}\\
        {-{\delta}_{321}}&
        {-{\delta}_{322}}&
        {\delta}_{121}-{\delta}_{332}&
        0&
        {\delta}_{122}&
        {\delta}_{132}\\
        {\delta}_{221}-{\delta}_{331}&
        -{\delta}_{121}+{\delta}_{222}-{\delta}_{332}&
        {\delta}_{131}+{\delta}_{232}-{\delta}_{333}&
        {-{\delta}_{122}}&
        0&
        {\delta}_{133}\\
        {\delta}_{231}&
        -{\delta}_{131}+{\delta}_{232}&
        {\delta}_{233}&
        {-{\delta}_{132}}&
        {-{\delta}_{133}}&
        0\\
        \end{smallmatrix}\right)\, .
\end{gather}

Here we abbreviated $\delta_{ijk} = \frac{\partial}{\partial x_i}\frac{\partial}{\partial e_j}\frac{\partial}{\partial e_k}$.

The point is that \emph{this matrix drops rank on Bateman seven tuples of points}. In other words they are contained in the associated
Pfaffian divisor $D \subset \PP V(1, \: 2)$:

\begin{proposition}\xlabel{pMorleyZeroOnBateman}
Consider the map 
\[
a \, : \, V(2, \: 0) \otimes V(1, \: 2) \mapsto V (0, \: 2)
\]
induced by $m$. Then for given $Q\in V(0, \: 2)$ and $C \in V(0, \: 3)$, hence given Bateman seven-tuple $b (Q^{\ast }, \: C) \in V(1, \: 2)$,
there is a dual quadric $Q'\in V(2, \: 0)$ in the kernel of our skew matrix, i.e. with $a(Q', \: b(Q^{\ast }, \: C)) =0$. In fact, one may take
$Q' = Q^{\ast }$. 
\end{proposition}

\begin{proof}
Look at the diagram
\[
\begin{CD}
V(2,0) \otimes (V(2, 0) \otimes V(0,3)) @>{\mathrm{id}\otimes b}>> V(2,0) \otimes V(1,2)\\
@V{(m\circ \alpha )\otimes \mathrm{id}}VV    @V{a}VV\\
V(2,1) \otimes V(0, 3)   @>{e}>>       V(0, 2)  
\end{CD}
\]
(where $\alpha \, :\, V(2,0) \otimes V(2, 0) \to \Lambda^2 V(2, 0)$ is simply anti-symmetrization). One checks
\begin{itemize}
\item
The map $e$ exists and is nontrivial since $V(0,2)$ is contained in $V(2,1) \otimes V(0, 3)$.
\item
The diagram commutes since there is just one non-zero map (up to a factor) from $V(2,0) \otimes (V(2, 0) \otimes V(0,3))$ to $V(0, 2)$: the
summand $V(0, 2)$ is contained with multiplicity $1$ in $V(2,0) \otimes (V(2, 0) \otimes V(0,3))$. To check this it suffices to see that $V(0,2)$
is not contained in $S^2 V(2, 0) \otimes V(0, 3)$ by the Littlewood-Richardson rule. 
\end{itemize}
Thus we can calculate $a(Q', \: b(Q^{\ast }, \: C))$ in another way now:
\begin{gather*}
a(Q', \: b(Q^{\ast }, \: C)) = e ( ( m\circ \alpha ) (Q' , Q^{\ast}), \: C)
\end{gather*}
which is zero for $Q' = Q^{\ast }$.
\end{proof}

\begin{corollary}\xlabel{cDimBateman}
Bateman seven-tuples form a $13$-dimensional family in $\PP V(1, 2)$. So there is an irreducible component of the Pfaffian hypersurface $D$, say
$K$, in which Bateman seven tuples are dense. 
\end{corollary}

\begin{proof}
Proposition \ref{pMorleyZeroOnBateman} says that the image of the pure tensors under $b\, :\, V(2, \: 0) \otimes V(0, 3) \to V(1, 2)$ is not all
of $V(1, 2)$. By upper semi-continuity of fibre dimension it thus suffices to find a pure tensor through which passes a $1$-dimensional fibre (counted
projectively, otherwise a two-dimensional fibre). This is easily accomplished as follows. Fix
\[
Q^* = e_1^2+e_2^2+e_3^2,\: C =  x_1^3+x_2^3-x_3^3-(x_1+x_2+x_3)^3\, .
\]
We compute the dimension of the tangent space to the fibre through the point $(Q^*, \: C)$ of the restriction of $b$ to pure tensors (at the point $(Q^* ,
\: C)$). Let $R^* \in V(2, \: 0)$, $D \in V(0, \: 3)$ be arbitrary, and with $\epsilon$ such that $\epsilon^2 =0$ note that the equation 
\begin{gather*}
b(Q^{\ast }, \: C) = b (Q^{\ast } + \epsilon R^{\ast }, \: C + \epsilon D )
\end{gather*}
is equivalent to
\[
\sum_{i=1}^3 \left( \frac{\partial }{\partial e_i } R^* \otimes \frac{\partial }{\partial x_i } C + \frac{\partial }{\partial e_i } Q^* \otimes
\frac{ \partial }{\partial x_i} D \right) = 0\quad \mathrm{in} \quad V(1,\: 2) . 
\]
To simplify, we will write $d_{R^*}(\cdot ) = b (R^{\ast } , \cdot )$, $d_{Q^*} (\cdot ) = b (Q^{\ast }, \cdot )$ (differentials w.r.t. these dual quadrics)
and $\mathrm{tr} = e_1 \otimes x_1 + e_2 \otimes x_2 + e_3 \otimes x_3$. Note that then the previous equation is equivalent to the existence of a linear
form $l \in (\CC^3)^{\ast }$ with 
\[
d_{R^*} C + l \cdot \mathrm{tr} = d_{Q^*} D
\]
and $d_{Q^*} (\cdot )$ may be identified with the usual differential
\[
d (\cdot ) = \sum_i \frac{\partial }{\partial x_i } (\cdot ) \otimes dx_i\, :\, S^{k} (\CC^3)^{\ast }\otimes \Lambda^l (\CC^3) \to S^{k-1} (\CC^3)^{\ast
}\otimes \Lambda^{l+1} (\CC^3)\, .
\]
Hence our sought-for tangent space is the space of $R^{\ast }$ and $l$ such that
\[
d_{Q^{\ast }} (d_{R^{\ast }} C + l \cdot \mathrm{tr} ) =0
\]
which a priori amounts to $8$ linear equations on the $9$-dimensional space of $(R^{\ast }, \: l)$ ($8$ equations because the equality lives in the kernel
of 
$d\, :\, \Lambda^2 (\CC^3) \otimes ( \CC^3)^{\ast} \to \CC$). But one can check (e.g. with \ttfamily Macaulay 2\rmfamily) that only $7$ of these are independent
whence our fibre through
$(Q^*, \: C)$ is
$2$-dimensional as desired. 
\end{proof}

\begin{remark}\xlabel{rConnectionToMorley}
It follows from the above discussion that $K$ coincides with the so-called Morley hypersurface of \cite{Mor} and \cite{O-S}.
\end{remark}

We are now in a position to prove

\begin{theorem}\xlabel{tBateman}
The moduli space $\mathfrak{M}_B$ of Bateman seven-tuples of points in $\PP^2$ is rational.
\end{theorem}

\begin{proof}

Consider the rational map
\begin{gather*}
\varphi \, : \, D \dasharrow \mathrm{Grass } (2, \: \mathrm{S}^2 (V))
\end{gather*}
assigning to $P \in D$ its kernel.

A general point of $\mathrm{Grass } (2, \: \mathrm{S}^2 (V))$ is
a pencil of (dual) quadrics in $\PP^2$ with $4$ distinct base points which in turn determine the pencil. Acting by
$\mathrm{SL}_3 (\CC )$ we may suppose that these points are the four vertices of a frame of reference in $\PP^2$ and so
$\mathrm{SL}_3 (\CC )$ acts with a dense orbit in $\mathrm{Grass } (2, \: \mathrm{S}^2 (V))$. The stabilizer of a generic point
in $\mathrm{Grass } (2, \: \mathrm{S}^2 (V) )$ inside $\mathrm{SL}_3 (\CC )$ is the symmetric group $\mathfrak{S}_4$.

\

CLAIM: The Morley hypersurface $K$ is the unique irreducible component of $D$ which dominates $\mathrm{Grass } (2, \: \mathrm{S}^2 (V))$.

\

To prove the Claim, it suffices to exhibit  a Bateman seven-tuple of points such that the kernel of the skew-matrix determined
by this and by $m$ consists of two general conics. One can check  that for our previous choice of $Q^{\ast }$ and $C$
from Corollary \ref{cDimBateman}
\[ Q^* = e_1^2+e_2^2+e_3^2,\: C =  x_1^3+x_2^3-x_3^3-(x_1+x_2+x_3)^3\]
we get a Bateman $7$-tuple which when substituted into $M$ of formula \ref{FormulaM} yields the pencil 
\[ \langle e_1^2+e_2^2+e_3^2,\:  {e}_{1}^{2}-{e}_{1} {e}_{2}+{e}_{2}^{2}+{e}_{1} {e}_{3}+{e}_{2}
       {e}_{3}\rangle \]
with smooth base locus, i.e. $(1:\pm i\sqrt{2}:1)$ and $(\pm i\sqrt{2}:1:1)$.
 
\

Hence $F/\mathfrak{S}_4$, where $F$ is the generic fibre of $\varphi$, is birational to $\mathfrak{M}_B$ by Corollary \ref{cDimBateman} and the
preceding discussion. We will determine $F$ as $\mathfrak{S}_4$-representation. The irreducible representations of $\mathfrak{S}_4$ are $1$
(trivial),
$\epsilon$ (sign representation), 
$V_2$ ($2$-dimensional induced from permutation representation of $\mathfrak{S}_3$), $V_3$ ($3$-dimensional permutation representation),
$V_3' = V_3 \otimes\epsilon$. Then
\[
S^2 (V) = V_3 + V_2 + 1
\]
whence
\[
\Lambda^2 V(2, \: 0) = \Lambda^2 (V_3 + V_2 + 1) = \Lambda^2 V_2 + V_2 \otimes (V_3 +1) + \Lambda^2 (V_3 +1 )\, .
\]
We get $F = \Lambda^2 (V_3 +1) = \Lambda^2 V_3 + V_3 \otimes 1 = V_3 + V_3$. $\PP (V_3 + V_3)/\mathfrak{S}_4$ is rational since it is
generically a vector bundle over $\PP (V_3 )/\mathfrak{S}_4$ which is a unirational surface. This completes the proof of the Theorem.
\end{proof}

\end{document}